\documentclass[12pt, reqno]{amsart}
\usepackage{amsmath, amsthm, amscd, amsfonts, latexsym, amssymb, graphicx, color, enumerate}
\usepackage[colorlinks,linkcolor=,citecolor = blue]{hyperref}

\makeatletter \oddsidemargin.9375in \evensidemargin \oddsidemargin
\newtheorem{theorem}{Theorem}[section]
\newtheorem{lemma}[theorem]{Lemma}
\newtheorem{proposition}[theorem]{Proposition}
\newtheorem{corollary}[theorem]{Corollary}
\theoremstyle{definition}

\newtheorem{example}[theorem]{Example}

\theoremstyle{remark}

\numberwithin{equation}{section}

\newcommand{\pd}[1]{\mathfrak{#1}}
\newcommand{\ppdots}[1]{{{\mathfrak{p}}_1}, \dots ,{{\mathfrak{p}}_{#1}}}
\newcommand{\ass}[1]{\mathrm{Ass}(#1)}
\newcommand{\supp}[1]{\mathrm{Supp}(#1)}
\newcommand{\ann}[1]{\mathrm{Ann}(#1)}
\newcommand{\ppp}[1]{{{\mathfrak{p}}_1}^{r_1} \cdots {{\mathfrak{p}}_{#1}}^{r_{#1}}}

\newcommand{\ppr}[2]{{{\mathfrak{p}}_{#1}}^{r_{#1}} \cdots {{\mathfrak{p}}_{#2}}^{r_{#2}}}
\newcommand{\pp}[1]{{{\mathfrak{p}}_1} \cdots {{\mathfrak{p}}_{#1}}}

\begin{document}

\setcounter{page}{1}

\title[Ideals as prime factorization of submodules]{Ideals as generalized prime ideal factorization of submodules}
\author[Thulasi, Duraivel, and Mangayarcarassy]{K. R. Thulasi$^{*}$, T. Duraivel, and S. Mangayarcarassy}
\thanks{{The first author was supported by INSPIRE Fellowship (IF170488) of the Department of Science and Technology (DST), Government of India.\\}{\scriptsize
\hskip -0.4 true cm MSC(2010): Primary: 13A05; Secondary: 13A15, 13E05, 13E15
$*$Corresponding author }}
\begin{abstract}
For a submodule $N$ of an $R$-module $M$, a unique product of prime ideals in $R$ is assigned, which is called the generalized prime ideal factorization of $N$ in $M$, and denoted as ${\mathcal{P}}_M(N)$. But for a product of prime ideals ${{{\mathfrak{p}}_1} \cdots {{\mathfrak{p}}_{n}}}$ in $R$ and an $R$-module $M$, there may not exist a submodule $N$ in $M$ with ${\mathcal{P}}_{M}(N) = {{{\mathfrak{p}}_1} \cdots {{\mathfrak{p}}_{n}}}$. In this article, for an arbitrary product of prime ideals ${{{\mathfrak{p}}_1} \cdots {{\mathfrak{p}}_{n}}}$ and a module $M$, we find conditions for the existence of submodules in $M$ having ${{{\mathfrak{p}}_1} \cdots {{\mathfrak{p}}_{n}}}$ as their generalized prime ideal factorization.
\end{abstract}

\maketitle

\section{Introduction}

Throughout this article, $R$ denotes a commutative Noetherian ring with identity and $M$ will be a finitely generated unitary $R$-module. The reference for standard terminology and notations will be \cite{C} and \cite{D}.

A proper submodule $N$ of an $R$-module $M$ is called a prime submodule of $M$ if for any $a \in R$ and $x \in M$, $ax \in N$ implies $a \in (N:M)$ or $x \in N$. If $N$ is a prime submodule of $M$, then $(N : M) = \pd p$, a prime ideal in $R$, and we say $N$ is a $\pd p$-prime submodule of $M$. Let $N$ and $K$ be submodules of $M$. Then $K$ is called a $\pd p$-prime extension of $N$ in $M$ if $N$ is a $\pd p$-prime submodule of $K$, and it is denoted as $N \overset{\pd p} \subset K$. In this case, $\ass{K/N} = \{ \pd p \}$.

Let $N$ be a proper submodule of an $R$-module $M$. Then we have $\pd p \in \ass{M/N}$ if and only if there exists a $\pd p$-prime extension of $N$ in $M$ \cite[Lemma~3]{A}. A $\pd p$-prime extension $K$ of $N$ is said to be maximal if $K$ is maximal among the submodules of $M$ which are $\pd p$-prime extensions of $N$ in $M$. Since $M$ is Noetherian, maximal $\pd p$-prime extensions exist. A filtration of submodules $\mathcal{F} : N = M_0 \overset{{\pd p}_1}\subset M_1 \subset \cdots \subset M_{n-1} \overset{{\pd p}_n}\subset M_n = M$ is called a maximal prime extension (MPE) filtration of $M$ over $N$, if $M_{i-1}\overset{{\pd p}_i} \subset M_i$ is a maximal ${\pd p}_i$-prime extension in $M$ for $ 1 \leq i \leq n$. It is proved that $\ass{M/M_{i-1}} = \{{\pd p}_i,\ldots,{\pd p}_n\}$ for each $1 \leq i \leq n$ \cite[Proposition~14]{A}. Hence, the set of prime ideals which occur in any MPE filtration of $M$ over $N$ is exactly equal to $\ass{M/N}$.

A maximal $\pd p$-prime extension $K$ of $N$ is said to be regular if $\pd p$ is a maximal element in $\ass{M/N}$, and the filtration $\mathcal{F} : N = M_0 \overset{{\pd p}_1}\subset M_1 \subset \cdots \subset M_{n-1} \overset{{\pd p}_n}\subset M_n = M$ is called a regular prime extension (RPE) filtration of $M$ over $N$ if $M_{i-1}\overset{{\pd p}_i} \subset M_i$ is a regular ${\pd p}_i$-prime extension in $M$ for $ 1 \leq i \leq n$. In this case, for each $i<j$, $M_i \overset{{\pd p}_{i+1}}\subset M_{i+1} \cdots \subset M_{j-1} \overset{{\pd p}_j}\subset M_j$ is also an RPE filtration of $M_j$ over $M_i$. Since RPE filtrations are MPE filtrations, $\ass{M_j/M_i} = \{{\pd p}_{i+1},\ldots,{\pd p}_j\}$ for $1 \leq i < j \leq n$. In particular, $\ass{M/N} = \{{\pd p}_1,\ldots,{\pd p}_n\}$.

The following lemma gives the condition for interchanging the occurrences of prime ideals in an RPE filtration.

\begin{lemma}\cite[Lemma~20]{A}\label{interchange}
Let $N$ be a proper submodule of $M$ and $N = M_0 \subset \cdots \subset M_{i-1}\overset{{\pd p}_i}\subset M_i \overset{{\pd p}_{i+1}}\subset M_{i+1}\subset \cdots \subset M_n = M $ be an RPE filtration of $M$ over $N$. If ${\pd p}_{i+1} \not\subseteq {\pd p}_{i}$ for some $i$, then there exists a submodule $K_i$ of $M$ such that $N = M_0 \subset \cdots \subset M_{i-1}\overset{{\pd p}_{i+1}}\subset K_i \overset{{\pd p}_{i}}\subset M_{i+1}\subset \cdots \subset M_n$ $=$ $M$ is an RPE filtration of $M$ over $N$.
\end{lemma}

RPE filtrations satisfy the following uniqueness property.
\begin{lemma}\cite[Theorem~22]{A}
For a proper submodule $N$ of $M$, the number of times a prime ideal $\pd p$ occurs in any RPE filtration of $M$ over $N$ is unique, and hence, any two RPE filtrations of $M$ over $N$ have the same length.
\end{lemma}

The submodules which occur in an RPE filtration are characterized as follows.
\begin{lemma}\cite[Lemma~3.1]{B}.
Let $N$ be a proper submodule of an $R$-module $M$. If $N = M_0 \overset{{\pd p}_1}\subset M_1 \subset \cdots \subset M_{n-1} \overset{{\pd p}_n}\subset M_n = M$ is an RPE filtration of $M$ over $N$, then $M_i = \{ x \in M \mid \pp i x \subseteq N \}$ for $ 1 \leq i \leq n$. \label{lemma1}
\end{lemma}

Hence, the product of prime ideals that occur in any two RPE filtrations of $M$ over $N$ is the same. This product is called the generalized prime ideal factorization of $N$ in $M$ and denoted as ${\mathcal{P}}_M(N)$ in \cite{E}, and sufficient conditions for ${\mathcal{P}}_{M}(\pp n M) = \pp n$ were found, where ${\pd p}_1 , \dots , {\pd p}_n$ are prime ideals in $R$ \cite[Theorem~2.14]{E}.

There may be products of prime ideals that are not the generalized prime ideal factorization of any submodule of a given module.
\begin{example}\label{exa1}
Let $R = \frac{k[x,y,z]}{(xy - z^2 , x^2 - yz)}$ and $\overline{x}, \overline{y}, \overline{z}$ denote the images of $x, y, z$ respectively in $R$. Let $\pd p$ be the prime ideal $(\overline{x}, \overline{z})$. Then $({\pd p}^2 : {\pd p}) = (\overline{x}, \overline{y}, \overline{z})$. Suppose there exists an ideal $\pd a$ in $R$ with ${\mathcal{P}}_{R}(\pd a) = {\pd p}^2$. Then there exists an RPE filtration $\pd a \overset{\pd p} \subset {\pd a}_1 \overset{\pd p} \subset R$ and therefore, $\ass{R/\pd a} = \{ \pd p \}$. By Lemma \ref{lemma1}, ${\pd a}_1 = (\pd a : \pd p)$, and since ${\pd p}^2 \subseteq \pd a$, we have $({\pd p}^2 : {\pd p}) \subseteq (\pd a : \pd p)$. Since $(\overline{x}, \overline{y}, \overline{z}) = ({\pd p}^2 : {\pd p}) \subseteq (\pd a : \pd p) = {\pd a}_1 \subsetneq R$ and $(\overline{x}, \overline{y}, \overline{z})$ is a maximal ideal, $(\overline{x}, \overline{y}, \overline{z}) = (\pd a : \pd p)$. Then $(\overline{x}, \overline{y}, \overline{z}) = (\pd a : p)$ for every $p \in \pd p \setminus \pd a$. This would imply that $(\overline{x}, \overline{y}, \overline{z}) \in \ass{R/\pd a}$, a contradiction. Therefore, an ideal $\pd a$ in $R$ cannot have ${\mathcal{P}}_{R}(\pd a) = {\pd p}^2$.
\end{example}

In this article, for a product of prime ideals $\pp n$ (${\pd p}_i$'s not necessarily distinct), we find conditions for the existence of submodules $N$ of $M$ with ${\mathcal{P}}_M(N) = \pp n$. We also give a necessary and sufficient condition for ${\mathcal{P}}_{M}(\pp n M) = \pp n$.
\section{Ideals as Generalized Prime Ideal Factorization of Submodules}
\begin{lemma}\label{minimal}
Let $N$ be a submodule of $M$ and $ {\pd p}_1, \dots , {\pd p}_r $ be some minimal prime ideals in $\ass{M/N}$. Then there exists a submodule $K$ of $M$ containing $N$ such that ${\mathcal{P}}_{M}(K) = \pp r$.
\end{lemma}
\begin{proof}
Let $N = M_0 \overset{{\pd q}_1}\subset M_1 \subset \cdots \subset M_{n-1} \overset{{\pd q}_n}\subset M_n = M$ be an RPE filtration of $M$ over $N$. Since $\{ {\pd q}_1, \dots , {\pd q}_n \} = \ass {M/N}$, for each $1 \leq i \leq r$, ${\pd p}_i = {\pd q}_j$ for some $j$. Since $\ppdots r$ are minimal, we can reorder ${\pd q}_1, \dots , {\pd q}_n $ such that ${\pd q}_j \not \subset {\pd q}_k$ for $1 \leq j<k \leq n$ and ${\pd q}_{n-r+i} = {\pd p}_i$ for $1 \leq i \leq r$. So using Lemma \ref{interchange} sufficient times we can have an RPE filtration $$N = K_0 \subset K_1 \subset \cdots \subset K_{n-r} \overset{{\pd p}_1} \subset K_{n-r+1} \overset{{\pd p}_2}\subset \cdots \subset K_{n-1} \overset{{\pd p}_r} \subset K_{n} = M$$ of $M$ over $N$. Then $$K_{n-r} \overset{{\pd p}_1} \subset K_{n-r+1} \overset{{\pd p}_2}\subset \cdots \subset K_{n-1} \overset{{\pd p}_r} \subset K_{n} = M$$ is an RPE filtration. So if $K= K_{n-r}$, then $K$ is a submodule of $M$ containing $N$ with ${\mathcal{P}}_{M}(K) = \pp r$.
\end{proof}

Now we show that for a prime ideal $\pd p$ in $R$, $\, \pd p \in \supp M$ is a necessary and sufficient condition for the existence of a submodule $N$ in $M$ with ${\mathcal{P}}_{M}(N) = \pd p$. More generally, we have the following result.

\begin{proposition}\label{proposition1}
Let $M$ be an $R$-module and ${\pd p}_1 , \dots , {\pd p}_n$ be prime ideals in $R$ such that ${\pd p}_i \not \subseteq {\pd p}_j$ for every $i, j \in \{ 1, \dots, n \}$ with $ i \neq j$. Then the following are equivalent:
\begin{enumerate}[(i)]
\item $\{{\pd p}_1 , \dots , {\pd p}_n\} \subseteq \supp{M}$;
\item ${\pd p}_i \in \supp{M/{\pp n M}}$ for every $1 \leq i \leq n$;
\item ${\pd p}_i$ is minimal in $\ass{M/{\pp n M}}$ for every $1 \leq i \leq n$;
\item There exists a submodule $N$ in $M$ with ${\mathcal{P}}_M(N) = \pp n$.
\end{enumerate}
\end{proposition}
\begin{proof}
(i) $\Rightarrow$ (ii): Suppose ${\pd p}_i \notin \supp{M/{\pp n M}}$ for some $i$. Then ${({M}/{{\pp n M}})}_{{\pd p}_i} = 0$. So we get $M_{{\pd p}_i} = (\pp n) M_{{\pd p}_i}$. Since $(\pp n)_{{\pd p}_i} \subseteq {\pd p}_iR_{{\pd p}_i}$, by Nakayama's lemma $M_{{\pd p}_i} = 0$. Therefore ${\pd p}_i \notin \supp{M}$.

(ii) $\Rightarrow$ (iii): If ${\pd q} \in \supp{M/{\pp n M}}$, then $\pp n \subseteq \pd q$, and therefore $\pd q$ contains some ${\pd p}_i$. So the set of minimal elements of $\supp{M/{\pp n M}}$ is contained in $\{{\pd p}_1 , \dots , {\pd p}_n\}$. Since ${\pd p}_i \not \subseteq {\pd p}_j$ for all $ i \neq j$, ${\pd p}_1 , \dots , {\pd p}_n$ are minimal elements in $\supp{M/{\pp n M}}$. Therefore ${\pd p}_1 , \dots , {\pd p}_n$ are minimal in $\ass{M/{\pp n M}}$.

(iii) $\Rightarrow$ (iv): Since ${\pd p}_1 , \dots , {\pd p}_n$ are minimal in $\ass{M/{\pp n M}}$, by Lemma \ref{minimal}, there exists a submodule $N$ of $M$ with ${\mathcal{P}}_M(N) = \pp n$.

(iv) $\Rightarrow$ (i): Since ${\pd p}_1 , \dots , {\pd p}_n$ are the prime ideals which occur in an RPE filtration of $M$ over $N$, $\{{\pd p}_1 , \dots , {\pd p}_n\} = \ass{M/N} \subseteq \supp{M}$.
\end{proof}

\begin{corollary}\label{cor2.1}
Let $\pd p$ be a prime ideal in $R$. Then $\pd p \in \supp{M}$ if and only if there exists a submodule $N$ in $M$ with ${\mathcal{P}}_{M}(N) = \pd p$.
\end{corollary}
In Proposition \ref{proposition1}, the prime ideals are distinct. Now we find conditions for the product of prime ideals that need not be distinct to be a generalized prime ideal factorization of some submodule.
\begin{proposition}\label{prop2.6}
Let $\pd p$ be a prime ideal in $R$ and $r$ be a positive integer. If $\pd p \in \ass{{{\pd p}^{r-1}M}/{{\pd p}^rM}}$, then there exists a submodule $N$ in $M$ such that ${\mathcal{P}}_{M}(N) = {\pd p}^r$.
\end{proposition}
\begin{proof}
Let $N = \{x \in M \mid ({\pd p}^r M : x) \not\subseteq \pd p \}$. Let $x_1, x_2 \in N$ and $u \in R$. Then there exists $a_1 \in ({\pd p}^r M : x_1) \setminus \pd p$ and $a_2 \in ({\pd p}^r M : x_2) \setminus \pd p$. Then $a_1a_2 \in ({\pd p}^r M : ux_1 + x_2) \setminus \pd p$, which implies that $ux_1 + x_2 \in N$. Hence $N$ is a submodule of $M$. Since $\pd p \in \ass{{{\pd p}^{r-1}M}/{{\pd p}^rM}}$, there exists $x \in {\pd p}^{r-1}M$ such that $\pd p = ({\pd p}^rM : x)$. This implies $x \notin N$. Therefore $N$ is a proper submodule of $M$. Also, $N \supseteq {\pd p}^r M$.

We claim that $\ass{M/N} = \{ \pd p \}$. Let $\pd q \in \ass{M/N}$. Then ${\pd p}^r \subseteq \pd q$ since ${\pd p}^r M \subseteq N$. Therefore ${\pd p} \subseteq \pd q$. Now $\pd q = (N : z)$ for some $z \in M$, $z \notin N$, that is, $({\pd p}^r M : z) \subseteq \pd p$. Let $a \in \pd q$. Then $az \in N$, which gives $({\pd p}^r M : az) \not\subseteq \pd p$. Let $b \in R \setminus \pd p$ such that $baz \in {\pd p}^r M$, i.e., $ba \in ({\pd p}^r M : z) \subseteq \pd p$. This implies $a \in \pd p$. Therefore $\pd q \subseteq \pd p$. Hence $\ass{M/N} = \{ \pd p \}$.

If $N = M_0 \overset{{\pd p}_1} \subset M_1  \subset \cdots \subset M_{k-1} \overset{{\pd p}_k} \subset M_k = M$ is an RPE filtration of $M$ over $N$, then $\{{\pd p}_1, \dots, {\pd p}_k \} = \ass{M/N}= \{ \pd p \}$. So ${\mathcal{P}}_{M}(N) = {\pd p}^k$. Suppose $k < r$. Then ${\pd p}^{r-1} \subseteq {\pd p}^k$, which implies ${\pd p}^{r-1}M \subseteq {\pd p}^k M \subseteq N$. So, for every $x \in {\pd p}^{r-1}M$, $({\pd p}^r M : x) \not\subseteq \pd p$. But $\pd p \in \ass{{{\pd p}^{r-1}M}/{{\pd p}^rM}}$ implies $\pd p = ({\pd p}^rM : x)$ for some $x \in {\pd p}^{r-1}M$, a contradiction. Therefore, $k \geq r$, and this implies $M_r \subseteq M_k = M$. By Lemma \ref{lemma1}, $M_r = \{x \in M \mid {\pd p}^r x \subseteq N \}$. For any $x \in M$, ${\pd p}^r x \subseteq {\pd p}^r M \subseteq N$. Therefore $M_r = M$. So, $N \overset{\pd p} \subset M_1 \overset{\pd p} \subset \cdots \overset{\pd p} \subset M_r = M $ is an RPE filtration of $M$ over $N$, and hence ${\mathcal{P}}_{M}(N) = {\pd p}^r$.
\end{proof}

In Example \ref{exa1}, $\pd p \notin \ass{{\pd p}/{\pd p}^2} = \{(\overline{x}, \overline{y}, \overline{z}) \}$. So $\pd p$ need not be an element of $\ass{{{\pd p}^{r-1}M}/{{\pd p}^rM}}$ even if ${{\pd p}^rM} \subsetneq {{\pd p}^{r-1}M}$. 

\begin{theorem}\label{mainresult2}
Let $M$ be an $R$-module, ${\pd p}_1, \dots, {\pd p}_n$ be distinct prime ideals in $R$ ordered as ${\pd p}_i \not \subset {\pd p}_j$ for $i < j$, and $r_1, \dots , r_n$ be positive integers. If ${\pd p}_i \in \supp{{{\pd p}_i}^{r_i - 1} {{\pd p}_{i+1}}^{r_{i+1}} \cdots {{\pd p}_n}^{r_n}M}$ for $i = 1, \dots , n$, then there exists a submodule $N$ in $M$ such that ${\mathcal{P}}_M(N) = \ppp n$.
\end{theorem}
\begin{proof}
We prove by induction on $n$. If $n=1$, ${\pd p}_1 \in \supp{{{\pd p}_1}^{r_1 - 1}M}$ and by Proposition \ref{proposition1}, ${\pd p}_1 \in \ass{{{{\pd p}_1}^{{r_1}-1}M}/{{{\pd p}_1}^{r_1}M}}$. Then by Proposition \ref{prop2.6}, there exists a submodule $N$ in $M$ with ${\mathcal{P}}_{M}(N) = {{\pd p}_1}^{r_1}$. Now let $n > 1$, and assume that the result is true for $n-1$ prime ideals. Then there exists a submodule $L$ in $M$ with ${\mathcal{P}}_M(L) = {{\pd p}_2}^{r_2} \cdots {{\pd p}_n}^{r_n}$. That is, we have an RPE filtration
\begin{equation}\label{eqnn2.4}
L \overset{{\pd p}_2}\subset L^{(2)}_1 \overset{{\pd p}_{2}}\subset \cdots \overset{{\pd p}_{2}}\subset L^{(2)}_{r_{2}} \overset{{\pd p}_{3}}\subset L^{(3)}_1 \overset{{\pd p}_{3}}\subset \cdots \subset L^{(n)}_{r_n - 1} \overset{{\pd p}_{n}}\subset L^{(n)}_{r_n} = M.
\end{equation}
Then ${{\pd p}_2}^{r_2} \cdots {{\pd p}_n}^{r_n}M \subseteq L$.

So, we have $\ann{{{\pd p}_1}^{r_1 - 1}L} \subseteq  \ann{{{\pd p}_1}^{r_1 - 1} {{\pd p}_2}^{r_2} \cdots {{\pd p}_n}^{r_n} M} \subseteq {\pd p}_1$ since ${\pd p}_1 \in \supp{ {{\pd p}_1}^{r_1 - 1}{{\pd p}_2}^{r_2} \cdots {{\pd p}_n}^{r_n} M}$. That is, ${\pd p}_1 \in \supp{{{\pd p}_1}^{r_1 - 1}L}$, and by Proposition \ref{proposition1}, ${\pd p}_1 \in \ass{{{\pd p}_1}^{r_1 - 1}L/{{\pd p}_1}^{r_1}L}$. Then by Proposition \ref{prop2.6}, there exists a submodule $N$ in $L$ such that ${\mathcal{P}}_{L}(N) = {{\pd p}_1}^{r_1}$. That is, we have the RPE filtration
\begin{equation}\label{eqn2.44}
N \overset{{\pd p}_{1}}\subset L^{(1)}_1 \overset{{\pd p}_{1}}\subset L^{(1)}_2 \subset \cdots \overset{{\pd p}_{1}}\subset L^{(1)}_{r_{1}} = L.
\end{equation}

Next, we show that
\begin{multline}\label{eqnn2.6}
N=L^{(1)}_0 \overset{{\pd p}_{1}}\subset L^{(1)}_1 \overset{{\pd p}_{1}}\subset L^{(1)}_2 \subset \cdots \overset{{\pd p}_{1}}\subset L^{(1)}_{r_{1}} = L \overset{{\pd p}_{2}}\subset L^{(2)}_1 \overset{{\pd p}_{2}}\subset \cdots \\ \overset{{\pd p}_{2}}\subset L^{(2)}_{r_{2}} \overset{{\pd p}_{3}}\subset L^{(3)}_1 \subset \cdots \overset{{\pd p}_{n-1}}\subset L^{(n-1)}_{r_{n-1}} \overset{{\pd p}_{n}}\subset L^{(n)}_1 \overset{{\pd p}_{n}}\subset \cdots \overset{{\pd p}_{n}}\subset L^{(n)}_{r_n} = M
\end{multline}
is an RPE filtration of $M$ over $N$, which would imply that ${\mathcal{P}}_M(N) = \ppp n$. Since the filtration (\ref{eqnn2.4}) is already an RPE filtration, it is enough to show that $L^{(1)}_{j-1} \overset{{\pd p}_1}\subset L^{(1)}_j$ is a regular prime extension in $M$ for $1 \leq j \leq r_1$.

From (\ref{eqn2.44}) we have that $L^{(1)}_{j-1} \subset L^{(1)}_j$ is a ${\pd p}_1$-prime extension for every $1 \leq j \leq r_1$. Suppose $L^{(1)}_{j-1} \overset{{\pd p}_1}\subset L^{(1)}_j$ is not a maximal ${\pd p}_1$-prime extension in $M$ for some $j$. Then there exists a submodule $K \supset L^{(1)}_j$ such that $L^{(1)}_{j-1} \overset{{\pd p}_1} \subset K$ is a ${\pd p}_1$-prime extension in $M$. Since $L^{(1)}_{j-1} \overset{{\pd p}_1}\subset L^{(1)}_j$ is a maximal ${\pd p}_1$-prime extension in $L$, $K \not \subseteq L$. Let $x \in K \setminus L$. For $2 \leq i \leq n$, since ${\pd p}_1 \not \subseteq {\pd p}_i$, there exists $p_i \in {\pd p}_1 \setminus {\pd p}_i$. Then $p_i x \in L^{(1)}_{j-1}$. Since $L^{(1)}_{j-1} \subset L$, from (\ref{eqnn2.4}) we get that $p_i x \in L^{(i)}_k$ for every $2 \leq i \leq n$, $1 \leq k \leq r_i$.

Since $p_n x \in L^{(n)}_{r_n - 1}$, $L^{(n)}_{r_n - 1} \subset M$ is a ${{\pd p}_{n}}$-prime extension, and $p_n \notin {\pd p}_n$, we have $x \in L^{(n)}_{r_n - 1}$. Then $p_n x \in L^{(n)}_{r_n - 2}$ and $L^{(n)}_{r_n - 2} \overset{{\pd p}_{n}}\subset L^{(n)}_{r_n - 1}$ is a prime extension implies $x \in L^{(n)}_{r_n - 2}$. Repeating this argument $r_n-3$ times, we get $x \in L^{(n-1)}_{r_{n-1}}$. 

Replacing $M$ by $L^{(n-1)}_{r_{n-1}}$ and $p_n$ by $p_{n-1}$ in the previous paragraph, we get $x \in L^{(n-2)}_{r_{n-2}}$. Continuing this process, finally we get $x \in L^{(1)}_{r_{1}} = L$, a contradiction. Therefore, $L^{(1)}_{j-1} \overset{{\pd p}_1}\subset L^{(1)}_j$ is a maximal prime extension in $M$ for every $1 \leq j \leq r_1$, and hence (\ref{eqnn2.6}) is an MPE filtration of $M$ over $N$.

So, for $1 \leq j \leq r_1$, we get $\ass{M/L^{(1)}_{j-1}} = \{ {\pd p}_1 , \dots , {\pd p}_n \}$ and since ${\pd p}_1 \not \subset {\pd p}_i$ for every $i>1$, ${\pd p}_1$ is maximal in $\ass{M/L^{(1)}_{j-1}}$. Therefore (\ref{eqnn2.6}) is an RPE filtration of $M$ over $N$. Hence ${\mathcal{P}}_M(N) = \ppp n$.
\end{proof}

The converse of Theorem \ref{mainresult2} does not hold. For example, if ${\pd p}_2 \subsetneq {\pd p}_1$ are prime ideals in a ring $R$ and $M$ is the $R$-module $\frac{R}{{\pd p}_2} \oplus \frac{R}{{\pd p}_2}$, then for its submodule $N = \frac{{\pd p}_1}{{\pd p}_2} \oplus 0$, we have the RPE filtration $$N = \dfrac{{\pd p}_1}{{\pd p}_2} \oplus 0 \quad \overset{{\pd p}_1} \subset \quad \dfrac{R}{{\pd p}_2} \oplus 0 \quad \overset{{\pd p}_2}\subset \quad \dfrac{R}{{\pd p}_2} \oplus \dfrac{R}{{\pd p}_2} = M$$ of $M$ over $N$. So we have ${\mathcal{P}}_M(N) = {\pd p}_1 {\pd p}_2$. But ${\pd p}_2 M = 0$. Therefore ${\pd p}_1 \notin \supp{{\pd p}_2 M}$.

Next, we show that if we assume further that ${\pd p}_i \not \subset {\pd p}_j$ for $i \neq j$, the converse of Theorem \ref{mainresult2} holds. We need the following lemma.
\begin{lemma}\cite[Lemma~2.8]{B}
If $N \overset{\pd p} \subset K$ is a regular $\pd p$-prime extension in $M$, then for any submodule $L$ of $M$, $N \cap L \overset{\pd p} \subset K \cap L$ is a regular $\pd p$-prime extension in $L$ when $N \cap L \neq K \cap L$. \label{lemma}
\end{lemma}
\begin{theorem}\label{converse}
Let $N$ be a submodule of $M$ with ${\mathcal{P}}_M(N) = \ppp n$, where ${\pd p}_1, \dots, {\pd p}_n$ are distinct prime ideals in $R$ and $r_1, \dots , r_n$ are positive integers. If all the prime ideals in $\ass{M/N}$ are minimal, then ${\pd p}_i \in \supp{\ppp {i-1} {{\pd p}_i}^{r_i - 1} \ppr {i+1} nM}$ for $i = 1, \dots, n$.
\end{theorem}
\begin{proof}
Since ${\pd p}_1, \dots, {\pd p}_n$ are minimal, for every $i$ we can reorder ${\pd p}_1, \dots, {\pd p}_n$ such that ${\pd p}_1 = {\pd p}_i$ and by Lemma \ref{interchange}, we have an RPE filtration $$N \overset{{\pd p}_{1}}\subset L^{(1)}_1 \overset{{\pd p}_{1}}\subset L^{(1)}_2 \subset \cdots \subset L^{(1)}_{r_1 - 1} \overset{{\pd p}_{1}}\subset L^{(1)}_{r_{1}}\overset{{\pd p}_{2}}\subset L^{(2)}_1 \subset \cdots \overset{{\pd p}_{n}}\subset L^{(n)}_{r_n} = M$$ of $M$ over $N$. So it is enough to show that ${\pd p}_1 \in \supp{{{\pd p}_1}^{r_1 - 1}$ ${{\pd p}_2}^{r_2}$ $\cdots$ ${{\pd p}_n}^{r_n} M}$. Clearly $\ppp nM \subseteq N$ and ${{{\pd p}_1}^{r_1 - 1}{{\pd p}_2}^{r_2} \cdots {{\pd p}_n}^{r_n} M} \subseteq L^{(1)}_1$.

We claim that ${{\pd p}_1}^{r_1 - 1}{{\pd p}_2}^{r_2} \cdots {{\pd p}_n}^{r_n} M \not \subseteq N$. Let $x \in L^{(1)}_{r_1} \setminus L^{(1)}_{r_1 - 1}$. Then ${{\pd p}_1}^{r_1 - 1} x \subseteq L^{(1)}_1$ and ${{\pd p}_1}^{r_1 - 1} x \not\subset N$. So there exists $b \in {{\pd p}_1}^{r_1 - 1}$ such that $bx \in L^{(1)}_1 \setminus N$. Choose $a_j \in {\pd p}_j \setminus {\pd p}_1$ for every $2 \leq j \leq n$ and let $a = \prod_{2 \leq j \leq n}a_j^{r_j}$. Then $bax \in {{{\pd p}_1}^{r_1 - 1}{{\pd p}_2}^{r_2} \cdots {{\pd p}_n}^{r_n} M}$. Suppose $bax \in N$. Then, since $bx \in L^{(1)}_1 \setminus N$ and $N \overset{{\pd p}_1} \subset L^{(1)}_1$ is a ${\pd p}_1$-prime extension, we get $a \in {\pd p}_1$, a contradiction. So $bax \notin N$. Therefore ${{\pd p}_1}^{r_1 - 1}{{\pd p}_2}^{r_2} \cdots {{\pd p}_n}^{r_n} M \not \subseteq N$. So we have
\begin{align*}
N\, \cap \;({{\pd p}_1}^{r_1 - 1}{{\pd p}_2}^{r_2} \cdots {{\pd p}_n}^{r_n} M)\; &\subsetneq \;{{\pd p}_1}^{r_1 - 1}{{\pd p}_2}^{r_2} \cdots {{\pd p}_n}^{r_n} M \\
 &= \;L^{(1)}_1\, \cap \;({{{\pd p}_1}^{r_1 - 1}{{\pd p}_2}^{r_2} \cdots {{\pd p}_n}^{r_n} M}).
\end{align*}
Then by Lemma \ref{lemma}, $$N \, \cap \; ({{\pd p}_1}^{r_1 - 1}{{\pd p}_2}^{r_2} \cdots {{\pd p}_n}^{r_n} M) \quad \overset{{\pd p}_1} \subset \quad {{{\pd p}_1}^{r_1 - 1}{{\pd p}_2}^{r_2} \cdots {{\pd p}_n}^{r_n} M}$$ is a regular ${\pd p}_1$-prime extension in ${{{\pd p}_1}^{r_1 - 1}{{\pd p}_2}^{r_2} \cdots {{\pd p}_n}^{r_n} M}$. Then by Corollary \ref{cor2.1}, ${\pd p}_1 \in \supp{{{\pd p}_1}^{r_1 - 1}{{\pd p}_2}^{r_2} \cdots {{\pd p}_n}^{r_n} M}$.
\end{proof}
From Theorems \ref{mainresult2} and \ref{converse}, we get the following corollary.
\begin{corollary}
Let ${\pd p}_1 , \dots , {\pd p}_n$ be distinct prime ideals in $R$ with ${\pd p}_i \not \subset {\pd p}_j$ for $i \neq j$, and $r_1, \dots , r_n$ be positive integers. Then $\ppp n$ is the generalized prime ideal factorization of some submodule of $M$ if and only if ${\pd p}_i \in \supp{\ppp {i-1} {{\pd p}_i}^{r_i - 1} \ppr {i+1} nM}$ for every $1 \leq i \leq n$.
\end{corollary}

In \cite{E} we have found conditions for ${\mathcal{P}}_M(\pp nM) = \pp n$ \cite[Theorem~2.14]{E} and showed that this need not always be true \cite[Example~2.5]{E}. Now for an $R$-module $M$ and a product of prime ideals ${\pd a} = \pp n$ (${\pd p}_i$'s not necessarily distinct), we give a necessary and sufficient condition for ${\mathcal{P}}_M({\pd a}M) = {\pd a}$.
\begin{theorem}\label{iff-thm}
Let $M$ be an $R$-module and ${\pd p}_1, \dots, {\pd p}_n$ be prime ideals in $R$, not necessarily distinct, with ${\pd p}_i$ maximal among $\{{\pd p}_i, \dots, {\pd p}_n \}$ for $1 \leq i \leq n$. Let ${\pd a} = \pp n$, ${\pd a}_0 = R$, and ${\pd a}_i = \pp i$ \,for $i = 1, \dots , n-1$. Then ${\mathcal{P}}_M({\pd a}M) = {\pd a}$ if and only if $\mathrm{Ass}\big(\frac{({\pd a} M : {\pd a}_{i})}{({\pd a} M : {\pd a}_{i-1})} \big) = \{{\pd p}_i \}$ for every $1 \leq i \leq n$.
\end{theorem}
\begin{proof}
If $\mathrm{Ass}\big(\frac{({\pd a} M : {\pd a}_{i})}{({\pd a} M : {\pd a}_{i-1})} \big) = \{{\pd p}_i \}$ for every $1 \leq i \leq n$, we show that
\begin{equation}\label{eqn4}
{\pd a} M  \overset {{\pd p}_1} \subset ({\pd a} M : {\pd a}_{1}) \overset {{\pd p}_2} \subset ({\pd a} M : {\pd a}_{2}) \subset \cdots \subset ({\pd a} M : {\pd a}_{n-1}) \overset {{\pd p}_n} \subset ({\pd a} M : {\pd a})  = M
\end{equation}
is an RPE filtration.

$\mathrm{Ass}\big(\frac{({\pd a} M : {\pd a}_{i})}{({\pd a} M : {\pd a}_{i-1})} \big) = \{{\pd p}_i \}$ implies that there exists a regular ${\pd p}_i$-prime extension $K$ of $({\pd a} M : {\pd a}_{i-1})$ in $({\pd a} M : {\pd a}_{i})$. Then $K = \{x \in ({\pd a} M : {\pd a}_{i})\mid$ ${\pd p}_i x \subseteq ({\pd a} M : {\pd a}_{i-1}) \}$ by Lemma \ref{lemma1}. For every $x \in ({\pd a} M : {\pd a}_{i})$, ${\pd a}_{i-1}{\pd p}_i x = {\pd a}_i x \subseteq {\pd a}M$, that is, ${\pd p}_i x \subseteq ({\pd a} M : {\pd a}_{i-1})$. Therefore, $K = ({\pd a} M : {\pd a}_{i})$, and hence $({\pd a} M : {\pd a}_{i})$ is the unique regular ${\pd p}_i$-prime extension of $({\pd a} M : {\pd a}_{i-1})$ in $({\pd a} M : {\pd a}_{i})$. Suppose it is not maximal in $M$. Then there exists $x \in M \setminus ({\pd a} M : {\pd a}_{i})$ such that ${\pd p}_i x \subseteq ({\pd a} M : {\pd a}_{i-1})$, i.e., $x \in ({\pd a} M : {\pd a}_{i-1}{\pd p}_i) = ({\pd a} M : {\pd a}_{i})$, a contradiction. So $({\pd a} M : {\pd a}_{i})$ is a maximal ${\pd p}_i$-prime extension of $({\pd a} M : {\pd a}_{i-1})$ in $M$ for every $i$. Therefore (\ref{eqn4}) is an MPE filtration of $M$ over ${\pd a} M$. This implies that $\mathrm{Ass}\big(\frac{M}{({\pd a} M : {\pd a}_{i-1})} \big) = \{{\pd p}_i, \dots, {\pd p}_n \}$ for every $1 \leq i \leq n$. Since ${\pd p}_i$ is maximal among $\{{\pd p}_i, \dots, {\pd p}_n \}$, ${\pd p}_i$ is maximal in $\mathrm{Ass}\big(\frac{M}{({\pd a} M : {\pd a}_{i-1})} \big)$. Therefore (\ref{eqn4}) is an RPE filtration. Hence ${\mathcal{P}}_M({\pd a}M) = \pp n = {\pd a}$.

Conversely, suppose that ${\mathcal{P}}_M({\pd a}M) = {\pd a} = \pp n$. Since ${\pd p}_i$ is maximal among $\{{\pd p}_i, \dots, {\pd p}_n \}$ for every $1 \leq i \leq n$, we can construct an RPE filtration $${\pd a} M = M_0 \overset {{\pd p}_1} \subset M_{1} \overset {{\pd p}_2} \subset M_{2} \subset \cdots M_{n-1} \overset {{\pd p}_n} \subset M_n = M$$ of $M$ over ${\pd a} M$. Then by Lemma \ref{lemma1}, $M_{i} = \{ x \in M \mid \pp i x \subseteq {\pd a} M \}$, i.e., $M_{i} = ({\pd a} M : {\pd a}_i)$ for every $1 \leq i \leq n$. Then clearly $\mathrm{Ass}\big(\frac{({\pd a} M : {\pd a}_{i})}{({\pd a} M : {\pd a}_{i-1})} \big) = \mathrm{Ass}\big(\frac{M_{i}}{M_{i-1}} \big) = \{{\pd p}_i \}$ for every $1 \leq i \leq n$.
\end{proof}


%
\vskip 0.4 true cm

\bibliographystyle{amsplain}

\bigskip
\bigskip

\noindent {\footnotesize {\bf  K. R. Thulasi }\; \\
Department of Mathematics, Pondicherry University, Pondicherry, India.\\
 {\tt thulasi.3008@gmail.com}

\noindent {\footnotesize {\bf  T. Duraivel }\; \\
Department of Mathematics, Pondicherry University, Pondicherry, India.\\
 {\tt tduraivel@gmail.com}
 
\noindent {\footnotesize {\bf  S. Mangayarcarassy }\; \\
Department of Mathematics, Puducherry Technological University, Pondicherry, India.\\
 {\tt dmangay@pec.edu}

\end{document}